\documentclass[11pt]{amsart}
\usepackage{amsmath}
\usepackage{amsfonts}
\usepackage{amssymb}
\usepackage{amsthm}
\usepackage{color}
\usepackage{etoolbox}
\usepackage{graphicx}
\usepackage{subcaption}
\usepackage[mathscr]{euscript}

\usepackage{fancyhdr}
 
\theoremstyle{plain}
\newtheorem{theorem}{Theorem}[section]
\newtheorem{corollary}[theorem]{Corollary}
\newtheorem{lemma}[theorem]{Lemma}
\newtheorem{prop}[theorem]{Proposition}

\theoremstyle{definition}
\newtheorem{definition}[theorem]{Definition}

\begin{document}

\title{Enumerating the class of minimally path connected simplicial complexes}
\author{Lewis Mead}
\begin{abstract}
In 1983 Kalai proved an incredible generalisation of Cayley's formula for the number of trees on a labelled vertex set to a formula for a class of $r$-dimensional simplicial complexes. These simplicial complexes generalise trees by means of being homologically $\Bbb Q$-acyclic. In this text we consider a different generalisation of trees to the class of pure dimensional simplicial complexes that \textit{minimally connect a vertex set} (in the sense that the removal of any top dimensional face disconnects the complex). Our main result provides an upper and lower bound for the number of these minimally connected complexes on a labelled vertex set. We also prove that they are potentially vastly more topologically complex than the generalisation of Kalai. As an application of our bounds we compute the threshold probability for the connectivity of a random simplicial complex.
\end{abstract}
\maketitle
\section{Introduction}
Trees have a rich history of study dating back to the 1860s with their simple enumeration first given by Borchardt \cite{borchardt} that is now commonly referred to as Cayley's formula \cite{cayley}.
\vskip.1in
\noindent{\bf Cayley's formula.}
The number of trees on $n$ labelled vertices is $n^{n-2}$.
\vskip.1in
A tree may be uniquely characterised as being a connected and acyclic graph, both of these are topological properties that generalise naturally to higher dimensions and so it made sense to do so using the language of higher dimensional combinatorial structures -- i.e. simplicial complexes. This was done so in the groundbreaking paper of Kalai \cite{kalai} where he introduced \emph{$\Bbb Q$-acyclic simplicial complexes}.

\vskip.1in
\noindent{\bf Definition.}
$T$ is an $r$-dimensional \emph{$\Bbb Q$-acyclic simplicial complex} if $T$ is a simplicial complex with full $(r-1)$-dimensional skeleton with both $H_{r-1}(T;\Bbb Q) = 0$ and $H_r(T;\Bbb Q) = 0$.
\vskip.1in

We let $\mathcal{T}_{r}(n)$ denote the class of all such simplicial complexes on a labelled vertex set $[n] = \{1,\dots,n\}$. Kalai managed to find the following beautiful generalisation of Cayley's formula to this class of higher dimensional acyclic simplicial complexes.
\vskip.1in
\noindent{\bf Theorem of Kalai.}
Let $d < n$ be integers, then
$$\sum_{T\in \mathcal{T}_r(n)} \left|H_{r-1}(T;\Bbb Z)\right|^2 = n^{\binom{n-2}{r}}.$$
\vskip.1in
One need not generalise trees to higher dimensions in such a way however. Trees may also be uniquely characterised as connected graphs such that the removal of any edge disconnects it. This is a property that is certainly not generalised via the work the of Kalai. Generalising this notion of being minimally path connected in the sense of removing something and becoming disconnected was introduced by Schmidt-Pruzan and Shamir \cite{schmidt} where they used the language of hypergraphs. 

\vskip.1in
\noindent{\bf Definition.}
A \emph{hypertree} (or \emph{h-tree}) is a hypergraph $H = (V,E)$ that is connected and the removal of any edge from $E$ will disconnect $V$.
\vskip.1in
It is this flavour of generalisation that we will study in this text. We reformulate the definition of h-trees using the language of simplicial complexes -- the primary reason for this is that we will also be concerned with both the geometric and homological connectivity of these objects, making simplicial complexes the natural combinatorial framework to work with.

We call a simplicial complex an $r$-dimensional \emph{minimal connected cover} if it is connected and the removal of any $r$-dimensional simplex disconnects it (see Definition~\ref{def:mcc}).

Let $\mathcal{M}_r(n)$ denote the class of all such simplicial complexes on a labelled vertex set $[n] = \{1,\dots,n\}$. A primary goal of this text was to estimate the quantity $M_r(n):= |\mathcal{M}_r(n)|$. To this end we show the following (see Proposition~\ref{prop:lowerbound} and Corollary~\ref{cor:upperbound}):
\vskip.1in
\noindent{\bf Main Theorem.}
Fix $r\geq 1$ any integer. There exists constants $A,B > 0$ such that 
$$A^n\cdot n^n\leq M_r(n)\leq B^n\cdot n^n.$$
\vskip.1in
With the lower bound computed using the original work of \cite{schmidt} and the upper bound computed by relating minimal connected covers to a combinatorial object that have a known enumeration \cite{beineke}.

It's clear that $\Bbb Q$-acyclic simplicial complexes of Kalai are homologically connected up to dimension $r-2$ by the condition upon having full $(r-1)$-skeleton included, i.e. if $i\leq r-2$ then $H_i(T;\Bbb Z) = 0$ for all $T\in\mathcal{T}_r(n)$. The same is not necessarily true of minimal connected covers, in fact we show that any finite abelian group can be realised as a homology group of some minimal connected cover -- see Corollary~\ref{cor:anygroup}.

We conclude the text with what was the motivating example behind studying minimal connected covers, finding the threshold probability for a pure $r$-dimensional random simplicial complex to be connected. This is a topic that has been studied to great success by Cooley, Kang et al. \cite{cooley2,cooley, cooley3} in the last few years. These texts go far behind anything that we try to achieve here, studying thresholds for the vanishing of cohomology in different dimensions as well as what precisely happens within the phase transition itself. Here we will emulate the classical proof of Erd\H{o}s and R\'{e}nyi for computing the threshold for connectivity of a random graph that utilises in a fundamental way Cayley's formula, the bounds on $M_r(n)$ will play a similarly crucial role -- see Theorem~\ref{thm:generalconn}.
\section{Structure of Minimal Connected Covers}
\begin{definition}
Let $X$ be a simplicial complex. We say that a simplex $\sigma\in X$ is a \textit{maximal simplex} (or a \textit{facet}) if for every $\tau\supseteq\sigma$ then $\tau\in X$ if and only if $\tau = \sigma$. That is, no larger simplexes in $X$ contain $\sigma$.
\end{definition}
For the purposes of this text whenever we talk of maximal simplexes we will always assume they are of dimension at least $1$, i.e. we never consider isolated vertices to be maximal simplexes. We say that a simplicial complex $X$ is \textit{pure} if all maximal simplexes are of the dimension.
\begin{definition}
Let $X$ be a simplicial complex on vertex set $V = V(X)$. Let $M = M(X)$ be the set of maximal simplexes. We say that $X'$ is attained from $X$ by \textit{removing} a maximal simplex $\sigma\in M$ if $V(X') = V$ and $M(X') = M - \{\sigma\}$.
\end{definition}
\begin{definition}\label{def:mcc}
Let $Y$ be a pure $r$-dimensional simplicial complex on $[n]$. We call $Y$ a \textit{minimal connected cover} if $Y$ is connected and the removal of any facet disconnects $Y$.

Let $\mathcal{M}_r(n)$ denote the set of $r$-dimensional minimal connected covers on $[n]$ and let $M_r(n) = \left|\mathcal{M}_r(n)\right|$.
\end{definition}

\begin{figure}[h]\centering
\includegraphics[scale=0.9]{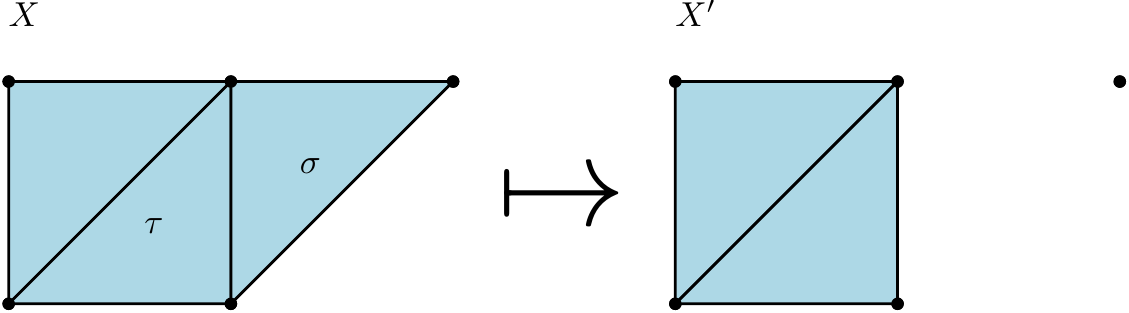}
\caption{A maximal simplex $\sigma$ is removed from a simplicial complex $X$. Observe that though removing $\sigma$ disconnects $X$ it is not true that $X$ a minimal connected cover as there exists a facet, $\tau$, that we could remove without disconnecting $X$.}
\end{figure}

\begin{prop}\label{prop:homconn}
For any $Y\in\mathcal{M}_r(n)$ one has
$$H_{r-1}(Y;\Bbb Z)_T = 0\quad\text{and}\quad H_r(Y;\Bbb Z) = 0.$$
\end{prop}
\begin{proof}
Every $r$-dimensional simplex $\sigma\in Y$ contains at least one free $(r-1)$-dimensional face, if this were not the case when one could remove such a $\sigma$ from $Y$ without affecting the path connectivity so $Y$ certainly could not have been minimally connected. We may therefore simplicially collapse every maximal $r$-simplex along this free face to obtain a new complex $Y'$ of dimension $r-1$ that is homotopy equivalent to $Y$, the statement then follows as $Y$ has homotopy dimension at most $r-1$.
\end{proof}
Notice the difference in homological behaviour of minimal connected covers compared with $\Bbb Q$-acyclic simplicial complexes of Kalai. Both require that the top dimensional homology $H_r$ vanishes, but the condition upon $H_{r-1}$ are completely opposing. As we shall see, in lower dimensions their differences increase even further.
\begin{prop}\label{prop:triangulable}
For any $k$-dimensional topological space $X$ with a finite triangulation there exists a minimal connected cover $Y\in \mathcal{M}_r(n)$ for any $r>k$ and some $n$ such that
$$Y\simeq X.$$
\end{prop}
\begin{proof}
Let $T$ be a triangulation of $X$. To every maximal simplex in $T$ of dimension $\ell$ we may simplicially join it to a new uniquely labelled simplex of dimension $r-\ell -1$, call this new simplicial complex $Y$. Clearly $Y$ simplicially collapses onto $T$, in particular $Y\simeq T$.
\end{proof}
\begin{figure}[h]\centering
\includegraphics[scale=0.7]{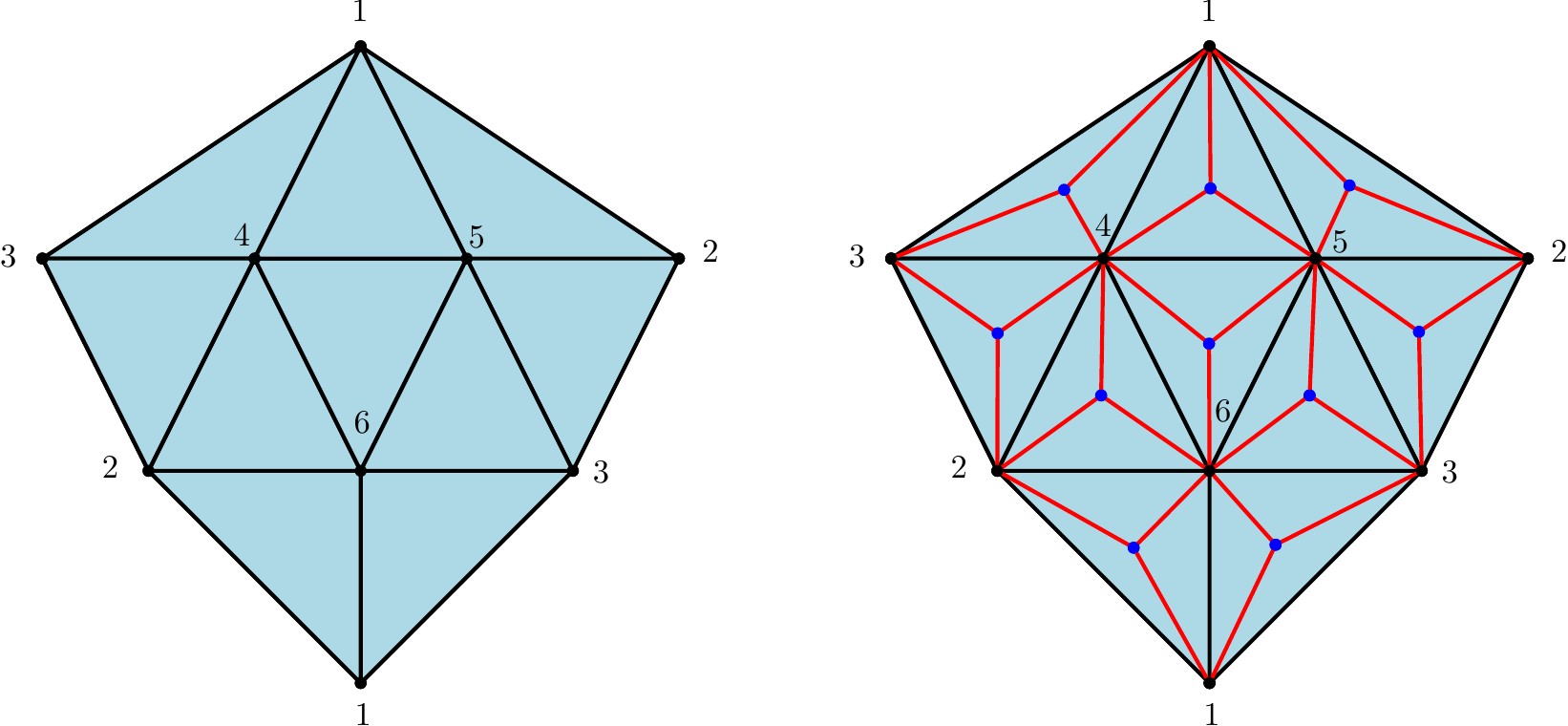}
\caption{Shows the process described in the proof of Proposition~\ref{prop:triangulable}. A triangulation of $\Bbb R\Bbb P^2$ is turned into a $3$-dimensional minimal connected cover.}
\label{fig:RP2}
\end{figure}
\begin{corollary}\label{cor:homfinite}
For any integer $r\geq 3$, any $m\geq 2$ and all $1\leq k\leq r-2$ there exists an $n$ such that there is a $Y\in\mathcal{M}_r(n)$ with $H_k(Y) = \Bbb Z_m$.
\end{corollary}
\begin{proof}
Consider the Moore space $M = M(\Bbb Z_m, k)$, that is a finite CW complex of dimension $k+1$ such that $H_k(M) = \Bbb Z_m$ and $H_i(M) = 0$ for all $i\neq k$.\footnote{$M$ is obtained from the sphere $S^k$ by attaching one $(k+1)$-cell by a map $S^k\to S^k$ of degree $m$.} One may then cover $M$ by open sets sufficiently finely so that the conditions of the Nerve Lemma (see Corollary 4G.3 of \cite{hatcher}) are satisfied, the obtained nerve complex of this cover $N$ has the same homotopy type as our Moore space $M$, in particular $H_k(N) = \Bbb Z_m$ and we may conclude by application of Proposition~\ref{prop:triangulable} to this $N$.
\end{proof}
In the proof of Corollary~\ref{cor:homfinite} one could alternatively consider the vertex minimal construction of a simplicial complex $X$ with prescribed torsion $H_{k-1}(X) = \Bbb Z_m$ as constructed in the paper Newman \cite{newman}. Perhaps with some work this same construction could be shown to give rise to the vertex minimal minimal connected cover with prescribed torsion.
\begin{corollary}\label{cor:homfree}
For all $r\geq 2$, any $1\leq i\leq r-1$ there exists an $n$ such that there is a $Y\in\mathcal{M}_r(n)$ with $H_i(Y) = \Bbb Z$.
\end{corollary}
\begin{proof}
Apply Proposition~\ref{prop:triangulable} to a triangulation of the $i$-sphere $S^i$.
\end{proof}
We make note of the following obvious but useful observation.
\begin{lemma}\label{lem:wedge}
Let $Y_1\in \mathcal{M}_r(n_1)$ and $Y_2\in \mathcal{M}_r(n_2)$ then their wedge along any two vertices is also a minimal connected cover, in particular $Y_1\vee Y_2\in\mathcal{M}_r(n_1+n_2-1)$.
\end{lemma}
\begin{corollary}\label{cor:anygroup}
Fix $r\geq 2$ and let $G$ be any finitely presented abelian group. Then for any $1\leq k\leq r-2$ there exists an $n$ such that there is a $Y\in\mathcal{M}_r(n)$ with $H_k(Y) = G$.
\end{corollary}
\begin{proof}
We may write $G = \Bbb Z^\alpha\oplus \Bbb Z_{m_1}\oplus\dots\oplus \Bbb Z_{m_\beta}$ by the Fundamental Theorem of Abelian Groups. The result then follows by applications of Corollary~\ref{cor:homfinite}, Corollary~\ref{cor:homfree}, and Lemma~\ref{lem:wedge}.
\end{proof}
\begin{lemma}\label{lemfacetnumber}
If $Y\in\mathcal{M}_r(n)$ then
$$\left\lceil\dfrac{n-1}{r}\right\rceil\leq f_r(Y)\leq n - r.$$
\end{lemma}
\begin{proof}
This follows from the following obvious inequalities relating the number of vertices and the number of facets
$$n\leq r f_r(Y) + 1\quad\text{and}\quad f+r \leq n.$$
\end{proof}
\begin{definition}
Let $Y\in\mathcal{M}_r(n)$ and let $v\in V(Y)$ be a vertex. We say that $v$ is a \textit{leaf} if it is contained in a unique $r$-simplex that we call the \textit{branch}. We say that $v$ is an \textit{external leaf} if removing its branch from $Y$ leaves a unique connected component and isolated vertices.
\end{definition}

\begin{lemma}\label{lemleaf}
Every $Y\in\mathcal{M}_r(n)$ contains an external leaf.
\end{lemma}
\begin{proof}
Suppose there are no external leaves. We will show by strong induction that this implies the existence of paths of $r$-simplexes of arbitrary length.

There certainly exists a path of length $1$, choose any facet in $Y$. Suppose there is a path of $r$-simplexes of length $k$ in $Y$, i.e. there exists $\sigma_1,\sigma_2,\dots,\sigma_k$ such that $\sigma_i\cap \sigma_{i+1}\neq \emptyset$ and $\sigma_i\cap\sigma_j = \emptyset$ for all $j\neq i, i+1$. If $\sigma_k$ had no neighbours except for $\sigma_{k-1}$ then $Y$ certainly contains an external leaf with branch $\sigma_k$. If all of the neighbours of $\sigma_k$ connect to some $\sigma_i$s then either $\sigma_k$ is not required for path connectivity or $\sigma_k$ is again a branch. This is a contradiction since $Y$ is a minimal connected cover, i.e. we are able to extend to a path of facets of length $k+1$. This holds true for $k > n-r$ which is a contradiction by Lemma~\ref{lemfacetnumber}, so $Y$ must contain an external leaf.
\end{proof}

\begin{lemma}\label{lem1skel}
The $1$-skeleton of a minimal connected cover determines it uniquely, i.e. if $Y,Z\in\mathcal{M}_r(n)$ with $Y^{(1)} = Z^{(1)}$ then $Y = Z$.
\end{lemma}
\begin{proof}
Suppose $Y^{(1)} = Z^{(1)}$ but $Y\neq Z$, then there exists a facet $\sigma$ which is in $Z$ but is not in $Y$. Such a $\sigma$ cannot be a branch in $Z$ otherwise it would have to be a branch in $Y$ as well, so $\sigma$ must be necessary for the connectivity of $Z$. That is, removing $\sigma$ from $Z$ must disconnect it and leave no isolated vertices. This can only occur if $\sigma$ contains at least one edge which is not contained within any other facet. This cannot happen by the assumption that our $1$-skeletons are the same.
\end{proof}
\section{Lower Bound}
\begin{definition}
We call a minimal connected cover $Y\in\mathcal{M}_r(n)$ \textit{treelike} if it is contractible and for every pair of distinct facets $\sigma,\tau$ one has $|\sigma\cap\tau| \leq 1$.
\end{definition}
The following is a restatement of Lemma 3.11 from Schmidt-Pruzan and Shamir \cite{schmidt}.
\begin{lemma}\label{lemhypertree}
Suppose $n = kr + 1$ for some integer $k$. Then the number of $Y\in \mathcal{M}_r(n)$ such that $Y$ is treelike equals
$$\dfrac{(n-1)!\cdot n^{k-1}}{k!\cdot r!^k}.$$
\end{lemma}
This result together with the following Lemma stating that $M_r(n)$ is a non-decreasing function of $n$ will give our lower bound.
\begin{lemma}\label{lemmonotonic}
$M_r(n) \leq M_r(n+1)$.
\end{lemma}
\begin{proof}
Let $Y\in\mathcal{M}_r(n)$, there exists an external leaf in $Y$ by Lemma~\ref{lemleaf}. Let $v$ be the smallest leaf in $Y$ with branch $\sigma$ with vertices $\{v,v_1,\dots, v_r\}$. Let $\sigma'$ denote a new $r$-simplex on vertex set $\{v_1,\dots,v_r, n+1\}$ and define a new simplicial complex $Y' = Y\cup \sigma'$.

It's clear that $Y'\in\mathcal{M}_r(n+1)$. Moreover one sees that $Y' = Z'$ if and only if $Y = Z$. We have therefore constructed an injective map $\mathcal{M}_r(n)\to\mathcal{M}_r(n+1)$ which proves the lemma.
\end{proof}

\begin{prop}\label{prop:lowerbound}
There exists a constant $A > 0$ such that
$$M_r(n) \geq A^n\cdot n^n.$$
\end{prop}
\begin{proof}
Let $n = kr + c$ for $c = 0,1,\dots, r-1$. Then by Lemma~\ref{lemmonotonic} and Lemma~\ref{lemhypertree}
\begin{align*}
M_r(n) \geq \dots \geq M_r(n - c + 1) &\geq \dfrac{(n-c)!\cdot (n-c+1)^{k-2}}{(k-1)!\cdot r!^{k-1}}\\
&\geq \left(\dfrac{n-c}{e\cdot r!}\right)^{n-c}\\
&\geq A^n \cdot n^n.
\end{align*}
Where in the final inequality we may take $A$ to be any constant less than $\frac{1}{2e r!}$.
\end{proof}
\section{Upper Bound}
\begin{definition}
Given an integer $r\geq 1$ an \textit{$r$-tree} is a graph which is defined inductively as follows:
\begin{itemize}
\item The complete graph on $r$ vertices $K_r$ is an $r$-tree.
\item Let $G$ be an $r$-tree on $n$ vertices, one may construct a new $r$-tree $G'$ on $n+1$ vertices by connecting a new vertex to any $r$ vertices that form a clique in $G$.
\end{itemize}
Any subgraph of an $r$-tree is called a \textit{partial $r$-tree}.

The following is a result of Beineke and Pippert \cite{beineke} for the enumeration of $r$-trees.
\begin{theorem}\label{thmrtrees}
There are $\binom{n}{r}\cdot\left[r(n-r)+1\right]^{n-r-2}$ labelled $r$-trees on $n$ vertices.
\end{theorem}
\end{definition}
\begin{prop}\label{propmccrtree}
If $Y\in\mathcal{M}_r(n)$ then $Y^{(1)}$ is a partial $r$-tree.
\end{prop}
\begin{proof}
We want to show that there exists an $r$-tree $G$ on $[n]$ such that $Y^{(1)}\subset G$.

We will prove this by strong induction on the number of vertices. If $Y\in\mathcal{M}_r(r+1)$ then $Y^{(1)} = K_{r+1}$ the complete graph, which is an $r$-tree. 

Now suppose that the $1$-skeleton of every minimal connected cover on less than $n$ vertices is a partial $r$-tree. Let $Y\in\mathcal{M}_r(n)$, by Lemma~\ref{lemleaf} there exists some external leaf with branch $\sigma$. When we remove $\sigma$ from $Y$ we are left with some $Y'\in\mathcal{M}_r(n-k)$ and $k$ isolated vertices for some $k = 1,\dots, r$ the number of leaves in the branch $\sigma$. Let $\tau = Y'\cap \sigma$ be the simplex of dimension $r-k$ and note that there exists an $(r-1)$-dimensional simplex $\tau'\in Y'$ with $\tau\subset \tau'$.

By our inductive hypothesis there exists an $r$-tree, $G$, on $n-k$ vertices such that $Y'^{(1)}\subset G$. If $k = 1$ then $G_1 = G\cup \sigma^{(1)}$ is an $r$-tree such that $Y^{(1)}\subset G_1$. If $k > 1$ let $\{v_1,\dots,v_k\}$ be the set of leaves and construct a new graph from $G$ as follows:
\begin{itemize}
\item Connect $v_1$ to all of the vertices in $\tau'$.
\item Connect $v_2$ to $v_1$ and any $r-1$ vertices in $\tau'$.
\item ...
\item Connect $v_j$ to all vertices $v_1,v_2,\dots, v_{j-1}$ and any $r-j+1$ vertices in $\tau'$.
\end{itemize}
Note that at each stage the new edges that are added ensure the graph is an $r$-tree. Moreover, the graph constructed at the $k$th step certainly contains $Y'^{(1)}\cup\sigma^{(1)}$ and thus it contains $Y^{(1)}$.
\end{proof}

\begin{figure}[h]
\centering
\begin{subfigure}{0.5\textwidth}
  \centering
  \includegraphics[width=.7\linewidth]{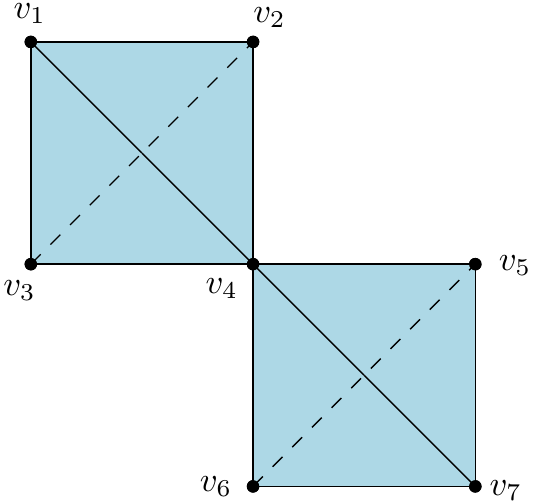}
\end{subfigure}
\begin{subfigure}{1\textwidth}
  \centering
  \includegraphics[width=0.7\linewidth]{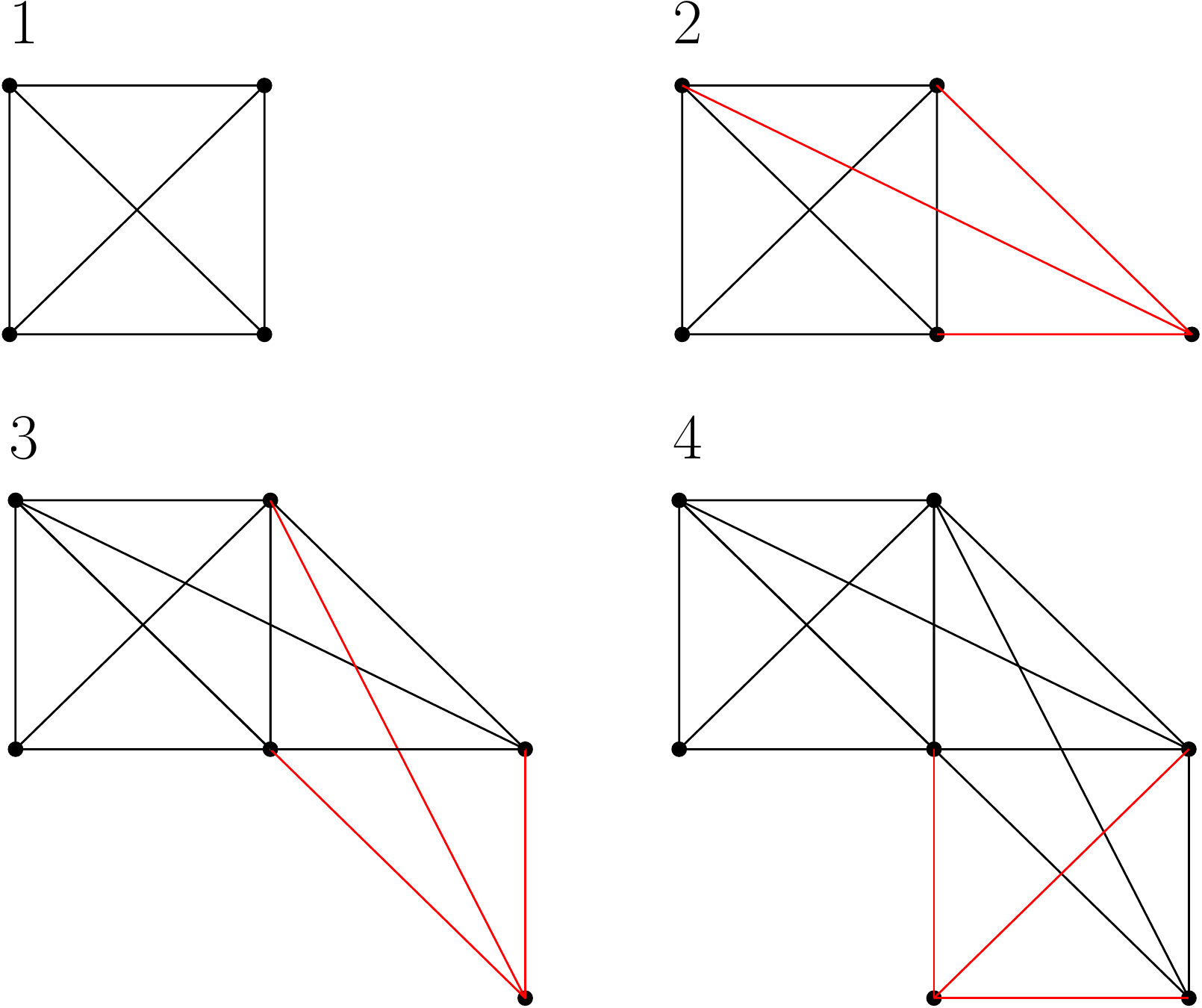}
\end{subfigure}
\caption{Shows the process described in the proof of Proposition~\ref{propmccrtree}. It shows that the $1$-skeleton of some $Y\in\mathcal{M}_3(7)$ (shown on the top) is a partial $3$-tree by removing the branch $\sigma = [v_4,v_5,v_6,v_7]$ and doing the described process with $\tau = [v_4]$ and $\tau' = [v_1,v_2,v_4]$.}
\label{fig:partial}
\end{figure}

\begin{corollary}\label{cor:upperbound}
There exists a constant $B>0$ such that
$$M_r(n)\leq B^n\cdot n^n.$$
\end{corollary}
\begin{proof}
Combining Lemma~\ref{lem1skel} and Proposition~\ref{propmccrtree} gives an injective map from $\mathcal{M}_r(n)$ to the set of partial $r$-trees on $n$ vertices, so we just need a bound on the size of this set and we are done.

It's clear that a partial $r$-tree on $n$ vertices has $\binom{r}{2} + r(n-r)$ edges. Therefore by Theorem~\ref{thmrtrees} we know that the number of partial $r$-trees on $n$ vertices equals
\begin{align*}
2^{\binom{r}{2} + r(n-r)}\cdot \binom{n}{r}\cdot\left[r(n-r)+1\right]^{n-r-2}
&\leq 2^{\binom{r}{2}}\cdot \left(2^r\cdot r\cdot n\right)^{n-2}\\
&\leq B^n\cdot n^n.
\end{align*}
Where in the final inequality we may take $B$ to be any constant greater than $2^{\binom{r+1}{2}}\cdot r$.
\end{proof}
\section{Application: The Threshold for Path Connectivity of Pure Random Simplicial Complexes}
In this section we consider an extended example that utilises the bound for the number of minimal connected covers found in Corollary~\ref{cor:upperbound}. The following is a classical result of Erd\H{o}s and R\'{e}nyi \cite{erdos} about connectivity in random graphs.
\begin{theorem}\label{thm:erdrenconn}
If $G = G(n,p)$ is an Erd\H{o}s-R\'{e}nyi random graph then $$p = \dfrac{\frac{1}{2}\log n}{n}$$ is the threshold probability for $G$ to have a unique connected component and isolated vertices.
\end{theorem}

Throughout, we will study the pure random simplicial complex $Y$ defined on the vertex set $[n] = \{1,\dots,n\}$ with each $r$-dimensional simplex included independently at random with probability $p$, this is a special case of the upper model random complexes studied in \cite{farber}. One may equivalently view this as a model for random $(r+1)$-uniform hypergraphs.

In contrast with classical random simplicial complexes of Linial, Meshulam, Wallach \cite{linial, meshulam} we observe that such a $Y$ has no condition requiring it to contain the full skeleton of dimension $(r-1)$ so questions about connectivity cannot be automatically taken as given. The recent paper of Cooley, Del Giudice, Kang, Spr\"ussel \cite{cooley} meticulously studies thresholds for homological connectivity of such a random simplicial complex $Y$ and goes far beyond the results presented in this section -- this section is not intended to provide any new results but to provide a proof of a result analogous to Theorem~\ref{thm:erdrenconn} using techniques similar to that of Erd\H{o}s and R\'{e}nyi.

\begin{theorem}\label{thm:generalconn}
Let $Y$ be the pure random simplicial complex on vertex set $[n]$ with each $r$-dimensional simplex included independently at random with probability $p$. Then $p = \dfrac{\frac{r!}{r+1}\log n}{n^r}$ is the threshold probability for $Y$ to have a unique connected component and potentially isolated vertices.
\end{theorem}
To show this we will first compute an upper bound for the expected number of connected components on $k$ vertices using our bound on the number of minimal connected covers (Lemma~\ref{lem:expec}), then we will compute the threshold probability for such a random complex $Y$ to have isolated $r$-simplexes (Lemma~\ref{lem:isolatedsimp}) before showing that this is precisely the threshold for the connectivity of $Y$ (Lemma~\ref{lem:uniquecomponent}).
\begin{lemma}\label{lem:expec}
The expected number of connected components on $k$ vertices in the random complex $Y$ is bounded above by 
$$C^{k}\cdot{{n}\choose k} k^{k} p^{\left\lceil\frac{k-1}{r}\right\rceil}(1-p)^{Q(n,k)}$$where $Q(n,k) = \sum_{i = 1}^r {k\choose i} {{n-k}\choose {r - i +1}}$ and $C$ is some fixed finite constant.
\end{lemma}
\begin{proof}
The probability that a given $k$ vertices form a connected component is the product of two probabilities: the probability that they are connected and the probability that they do not connect to anything outside. A given set of $k$ vertices $V$ is connected if and only if some minimal covering is present, which occurs with probability bounded above by $C^{k}k^{k} p^{\left\lceil\frac{k-1}{r}\right\rceil}$ by  Corollary~\ref{cor:upperbound} and Lemma~\ref{lemfacetnumber}.

Such a $V$ is disconnected from the rest of the complex if and only if no simplex with $1$, $2$, \dots, $r$ vertices in $V$ are selected. Therefore there must be
$$Q(n,k) = \sum_{i = 1}^r {k\choose i} {{n-k}\choose {r - i +1}}$$
$r$-simplexes with are not selected, which occurs with probability $(1-p)^{Q(n,k)}$. Thus the probability of one particular set of $k$ vertices defining a connected component is bounded above by $C^{k} k^{k} p^{\left\lceil\frac{k-1}{r}\right\rceil} (1-p)^{Q(n,k)}$. Therefore, the expected number of all such connected components on $k$ vertices is at most $C^{k}\cdot{{n}\choose k} k^{k} p^{\left\lceil\frac{k-1}{r}\right\rceil}(1-p)^{Q(n,k)}$.
\end{proof}

\begin{lemma}\label{lem:isolatedsimp}
Let $Y$ be the pure random simplicial complex on vertex set $[n]$ with each $r$-dimensional simplex included independently at random with probability $p$. Then $p = \dfrac{\frac{r!}{r+1}\cdot \log n}{n^r}$ is the threshold probability for the existence of isolated $r$-dimensional simplexes.
\end{lemma}
\begin{proof}
Let $N$ be the random variable which counts the number of isolated simplexes in $Y$. A simplex is isolated precisely when it is selected and no other simplexes with vertices in common are. Thus we must have 
\begin{align*}
Q(n,r+1) &= {{r+1}\choose 1}{{n-r-1}\choose r}+ \dots + {{r+1}\choose r}{{n-r-1}\choose 1}\\
&= \dfrac{r+1}{r!} n^r(1-O(1/n))
\end{align*}
simplexes that are not selected. So the probability of some simplex being isolated is
$p(1-p)^{Q(n,r+1)} = p(1-p)^{\frac{r+1}{r!} n^r(1-O(1/n))}$. The expected number of such isolated simplexes with $p = \frac{\alpha\log n}{n^r}$ is therefore given by
\begin{align*}
\Bbb E(N)&= {{n}\choose r+1}\cdot p(1-p)^{\frac{r+1}{r!} n^r(1-O(1/n))}\\
&\sim \dfrac{n^{r+1}}{(r+1)!}\cdot \dfrac{\alpha\log n}{n^r}\cdot \exp\left(-\dfrac{r+1}{r!}\cdot n^r\cdot \dfrac{\alpha\log n}{n^r}\cdot (1-O(1/n))\right)\\
& = \dfrac{\alpha n \log n}{(r+1)!}\cdot n^{\frac{-\alpha(r+1)}{r!}}\cdot n^{O(1/n)}\\
&\sim \dfrac{\alpha\log n}{(r+1)!}\cdot n^{1 - \frac{\alpha(r+1)}{r!}}.
\end{align*}
If $\alpha > \dfrac{r!}{r+1}$ then this expectation equals $o(1)$, so by Markov's inequality we see that such a random simplicial complex $Y$ has no isolated simplexes asymptotically almost surely.

Now suppose that $\alpha < \dfrac{r!}{r+1}$. The probability that two disjoint $r$-simplexes $\sigma$ and $\tau$ are both selected is 
$$p^2(1-p)^{2Q(n,r+1) - R(n,r)}$$
where $R(n,r) = \sum_{1\leq i,j\leq r} {{r+1}\choose i} {{r+1}\choose {j}}{{n-2r-2}\choose {r+1-i-j}} = O(n^{r-1})$. This counts all those simplexes which intersect both $\sigma$ and $\tau$ that we do not want to double count. Now
\begin{align*}
(1-p)^{R(n,r)} &> 1 - pR(n,r)\\ &= 1 -\frac{\alpha\log n}{n^r}O(n^{r-1})\\ &= 1 - O(\log n /n)\to 1.
\end{align*}
Therefore $(1-p)^{-R(n,r)} = 1+o(1)$ and so
\begin{align*}
\Bbb E(N^2)&\leq \Bbb E(N) + 2\sum_{\sigma\cap\tau = \emptyset} p^2(1-p)^{2Q(n,r+1) - R(n,r)}\\
&\leq \Bbb E(N) + 2\cdot {{n\choose {r+1}}\choose 2}\cdot p^2(1-p)^{2Q(n,r+1) - R(n,r)}\\
&\leq \Bbb E(N) + {n\choose {r+1}}^2\cdot p^2(1-p)^{2Q(n,r+1)}(1+o(1))\\
&= (1+o(1))\cdot\Bbb E(N)^2
\end{align*}
Where the last equality follows by using the fact that $$\Bbb E(N) \sim\dfrac{\alpha\log n}{(r+1)!}\cdot n^{1 - \frac{\alpha(r+1)}{r!}}\to\infty$$ for $\alpha < \dfrac{r!}{r+1}$. Therefore by using Chebychev's inequality in the form $$\Bbb P(N > 0) \geq \dfrac{\Bbb E(N)^2}{\Bbb E(N^2)}$$
we conclude that $Y$ has isolated simplexes with probability converging to one.
\end{proof}

%
In particular, Lemma~\ref{lem:isolatedsimp} tells us that if $\alpha < \frac{r+1}{r!}$ then the random simplicial complex in the description of Theorem~\ref{thm:generalconn} is disconnected. To complete the proof of Theorem~\ref{thm:generalconn} we will show the following,
\begin{lemma}\label{lem:uniquecomponent}
Let $Y$ be the pure random simplicial complex on vertex set $[n]$ with each $r$-dimensional simplex included independently at random with probability $p$. If $p = \dfrac{\alpha\log n}{n^r}$ with $\alpha > \frac{r!}{r+1}$ then $Y$ has a unique connected component and isolated vertices asymptotically almost surely.
\end{lemma}
\begin{proof}
We will prove this statement by showing that there are no connected components on $k$ vertices for all $r+1\leq k\leq n/2$, i.e. if this were true since there are no isolated $r$-dimensional simplexes there is a unique connected component supported by at least $n/2$ vertices and potentially isolated vertices proving the statement.

Let $X_k$ be the random variable which counts the number of connected components on $k$ vertices. We will show that for $p=\dfrac{\alpha\log n}{n^r}$ that the expected number of connected components of size $r+1\leq k\leq n/2$ is less than $O\left(n^{-\varepsilon}\cdot\log n\right)$ for some positive $\varepsilon$.

We use Lemma~\ref{lem:expec} to get the following upper bound
$$\Bbb E(X_k) \leq C^{k}{{n}\choose k} k^{k} p^{\frac{k-1}{r}} (1-p)^{Q(n,k)},$$
where we have used the trivial fact that $p^{\frac{k-1}{r}} \geq p^{\left\lceil\frac{k-1}{r}\right\rceil}$.

We will further simplify this bound by using the inequalities ${{n}\choose k}\leq \left(\frac{en}{k}\right)^k$, and $1-p\leq e^{-p}$. To complete the argument we need a better understanding of $Q(n,k) = \sum_{i = 1}^r {k\choose i} {{n-k}\choose {r - i +1}}$. For this we cite the result of Lemma~\ref{lem:Qmax} found in the Appendix which tells us that
$$\dfrac{r}{x} - \dfrac{\alpha\cdot r\cdot Q(n,x)}{n^r\cdot x}$$
is maximised by $x = r+1$ or $x = n/2$ in the domain $[r+1,n/2]$ for sufficiently large $n$ and that if $\alpha > \dfrac{r!}{r+1}$ then this maximal value is at most $-\varepsilon+ O(1/n)$ for some positive constant $\varepsilon$ dependent on $r$ and $\alpha$.

Putting this all together we get the following when we substitute $p=\frac{\alpha\log n}{n^r}$,
\begin{align*}
\Bbb E(X_k) &\leq C^{k}\cdot\left(\dfrac{en}{k}\right)^k \cdot k^k\cdot \left(\dfrac{\alpha\log n}{n^r}\right)^{\frac{k-1}{r}}\cdot\exp\left(\dfrac{-\alpha Q(n,k)\log n}{n^r}\right)\\
&\leq \alpha^{\frac{k}{r}}\cdot C^{k}\cdot e^k\cdot n^{1-\frac{\alpha Q(n,k)}{n^r}}\cdot\log^{\frac{k-1}{r}}n\\
&\leq \left(\mathrm{const}\cdot n^{r/k-\frac{\alpha r Q(n,k)}{n^r k}}\cdot\log n\right)^{k/r}\\
&= O\left(n^{-\varepsilon+O(1/n)}\cdot\log n\right)^{k/r}\\
&=O\left(n^{-\varepsilon}\cdot\log n\right)^{k/r}.
\end{align*}
Where the final line follows from the fact that $n^{O(1/n)}$ converges to a constant.

Since $\Bbb E(X_k)$ decreases geometrically we see by linearity that $\Bbb E\left(\sum_{k=r+1}^{n/2} X_k\right) = O\left(n^{-\varepsilon}\cdot\log n\right) = o(1)$. The result follows by application of Markov's inequality.
\end{proof}

Lemma~\ref{lem:isolatedsimp} and Lemma~\ref{lem:uniquecomponent} together prove Theorem~\ref{thm:generalconn}. It is a simple corollary to find the threshold probability for the connectedness of such a random simplicial complex, i.e. one just needs the threshold probability for the existence of isolated vertices.

\begin{lemma}\label{lem:isolatedvert}
Let $Y$ be the pure random simplicial complex on vertex set $[n]$ with each $r$-dimensional simplex included independently at random with probability $p$. Then $p = \dfrac{r! \log n}{n^r}$ is the threshold probability for the existence of isolated vertices.
\end{lemma}
\begin{proof}
A vertex in $Y$ is isolated with probability $q^{{n}\choose r}\sim \exp\left(-p{n\choose r}\right)$. Let $N$ be the random variable which counts the number of isolated vertices in $Y$. We see that
\begin{align*}
\Bbb E(N) &= n\cdot q^{{n}\choose r}\\
&\sim\exp\left(\log n - p{n\choose r}\right).
\end{align*}
Markov's inequality implies that we will have no isolated vertices with probability converging to one when this expectation tends to zero. This occurs precisely when $$\log n - p{n\choose r}\to -\infty,$$
which proves that for $p = \frac{r!\log n + \omega}{n^r}$ with $\omega\to\infty$ there does not exist any isolated vertices asymptotically almost surely.

For the second statement we will use the second moment method in the form of $\Bbb P(N > 0) \geq\dfrac{\Bbb E(N)^2}{\Bbb E(N^2)}$. Two distinct vertices are both isolated with probability $$(1-p)^{2{{n}\choose r} - {{{n-1}\choose r}}}.$$
Therefore,
\begin{align*}
\Bbb P(N>0) \geq \dfrac{\Bbb E(N)^2}{\Bbb E(N^2)} =\dfrac{n(1-p)^{{n}\choose r}}{1+(n-1)(1-p)^{{n-1}\choose r}}.
\end{align*}
For $p = \dfrac{r!\log n - \omega}{n^r}$ with $\omega\to\infty$ we observe that $(n-1)q^{{n-1}\choose r}\sim nq^{n\choose r}\sim\exp\left(\dfrac{\omega}{r!}\right)\to \infty$. It follows that $\Bbb P(X>0)\to 1$, so $Y$ 
\end{proof}
\begin{corollary}
Let $Y$ be the pure random simplicial complex on vertex set $[n]$ with each $r$-dimensional simplex included independently at random with probability $p$. Then $p = \dfrac{r! \log n}{n^r}$ is the threshold probability for $Y$ to be connected asymptotically almost surely.
\end{corollary}
\section{Acknowledgements}
The author would like to thank Andrew Newman for the ideas inspiring the proofs of Proposition~\ref{prop:homconn} and Corollary~\ref{cor:anygroup}, to Justin Ward for revealing the existence of $k$-trees, and to William Raynaud and Lewin Strauss for insightful discussions.
\newpage
\section{Appendix: Technical Lemma}
\begin{lemma}\label{lem:Qmax}
Let $r\geq 2$ be a fixed integer and define a function
$$Q_r(n,x):= \sum_{i=1}^r \binom{x}{i}\cdot \binom{n-x}{r+1-i},$$
we will think of $Q_r(n,x)$ as a polynomial in $(n, x)$ of bidegree $(r,r+1)$. Define a new function
$$q_r(n,x):= \dfrac{r}{x} - \dfrac{\alpha\cdot r\cdot Q_r(n,x)}{n^r\cdot x}$$
for some arbitrary positive constant $\alpha$.

The maxima of $q_r(n,x)$ over $[r+1,n/2]$ is attained at one of the two endpoints for sufficiently large $n$. Moreover, if $\alpha > \dfrac{r!}{r+1}$ and $n$ is sufficiently large then
$$\max_{x\in [r+1,n/2]} q_r(n,x) \leq -\varepsilon_r + O\left(1/n\right) < 0$$
where $\varepsilon_r = \min \left\{\dfrac{\alpha}{(r-1)!} - \dfrac{r}{r+1}, \sum{i=1}^r \frac{2\alpha r}{2^r i!(r+1-i)!}\right\}>0$.
\end{lemma}
\begin{proof}
The proof will rely on a few simple ideas. We will first show that there is just one $x\geq r+1$ for sufficiently large $n$ such that $q'_r(n,x) = 0$, i.e. a unique positive stationary point. We will then show that $q'_r(n,r+1) < 0$ for large enough $n$. This then implies that our stationary point is either a minima or point of inflection and moreover that on $[r+1,n/2]$ the maxima is attained at either $x = r+1$ or $x = n/2$, we then just need some good estimates for $q_r(n,x)$ at these points to conclude the proof.

A simple computation shows that
$$q'_r(n,x) = \dfrac{-r}{x^2} - \dfrac{\alpha r}{n^r}\cdot\left(\dfrac{xQ'_r - Q_r}{x^2}\right),$$
which is equal to zero iff
\begin{equation}\label{eqn:vanishing}
\dfrac{Q_r - xQ'_r}{n^r} = \dfrac{1}{\alpha}.
\end{equation}
Since $Q_r(n, 0) = 0$ we may write $Q_r(n,x) = x P_r(n,x)$ where $P_r(n,x)$ is a polynomial in $(n,x)$ of bidegree $(r,r)$ where the coefficients $a_{i,j}$ of terms $n^i x^j$ vanish if $i + j \geq r+1$.

Therefore,
$$Q_r - xQ'_r = -x^2 P'_r$$
where $P'_r$ is of bidegree $(r,r-1)$. In fact, since $a_{r,i}$ for all $i\geq 1$ we know that $P'_r$ is of bidegree $(r-1,r-1)$. Therefore, we may rewrite equation (\ref{eqn:vanishing}) as
\begin{equation}\label{eqnbij}
-x^2\cdot \sum_{i+j\leq r-1} b_{i,j}\dfrac{n^i x^j}{n^r} = \dfrac{1}{\alpha}
\end{equation}
where $b_{i,j} = (j+1)\cdot a_{i,j+1}$, i.e. we see that $b_{i,j} = 0$ for $i + j \geq r$. Observe that for $i+j\leq r-1$ and for $x\leq n/2$ $$\frac{x^j}{n^{r-i}}\leq \frac{1}{n}$$ so equation (\ref{eqnbij}) is of the form $O\left(\frac{x^2}{n}\right) =\text{constant}$. Therefore, for sufficiently large $n$ there is at most one positive $x$ satisfying (\ref{eqn:vanishing}), i.e. there is at most one stationary point of $q_r(n,x)$ in $[r+1,n/2]$.

We now compute
\begin{align*}
q'_r(n,r+1) &= \dfrac{-r}{(r+1)^2} - \dfrac{\alpha r}{n^r}\cdot\left(\dfrac{(r+1)Q'_r(n,r+1) - Q_r(n,r+1)}{(r+1)^2}\right)\\
&\sim \dfrac{-r}{(r+1)^2} - \alpha r\left(\dfrac{(r+1)/r! - (r+1)/r!}{(r+1)^2}\right)\\
&= \dfrac{-r}{(r+1)^2} < 0.
\end{align*}

That is for large enough $n$ we know that $q_r(n,x)$ is initially decreasing, so the critical point found above must be either a minima or a point of inflection. Therefore we know that the largest value of $q_r(n,x)$ on $[r+1,n/2]$ comes at one of the two endpoints.

A simple approximation shows $$Q_r(n,n/2) = \sum_{i=1}^r \dfrac{n^{r+1}}{2^r i!(r+1-i)!}\left(1+O(1/n)\right).$$
We let $A_r := \sum_{i=1}^r \frac{1}{2^r i!(r+1-i)!}$ and remark that $q_r(n,n/2) = -2\alpha r A_r +O(1/n)\leq -\varepsilon_r + O(1/n)$.

It is easy to see that
\begin{align*}
Q_r(n,r+1) &=  \sum_{i=1}^r {{r+1}\choose i}\cdot \left(\dfrac{n^{r+1-i}}{(r+1-i)!} + O(n^{r-i})\right)\\
&= \dfrac{(r+1)\cdot n^r}{r!} + O(n^{r-1}).
\end{align*}
Therefore we easily observe that
$$q_r(n,r+1) = \dfrac{r}{r+1} - \dfrac{\alpha}{(r-1)!}+ O(1/n)\leq -\varepsilon_r + O(1/n),$$
which completes the proof.
\end{proof}

\end{document}